\documentclass[a4paper,oneside,10pt, reqno]{amsproc}
\usepackage{amsthm}
\usepackage{amsfonts}
\usepackage{amsmath}
\usepackage{amssymb}
\usepackage{mathrsfs}
\usepackage{epsfig}
\usepackage{graphicx}
\usepackage{epstopdf}
\usepackage{tikz-cd}
\usepackage{color}
\usepackage{ifthen}
\usepackage{float}
\usepackage{tikz}
\usetikzlibrary{arrows}
\usetikzlibrary{calc}
\usepackage{mathtools}

\makeatletter
\@namedef{subjclassname@2010}{%
\textup{2010} Mathematics Subject Classification}
\makeatother

\newtheorem{theorem}{Theorem}[section]
\newtheorem{proposition}[theorem]{Proposition}
\newtheorem{corollary}[theorem]{Corollary}
\newtheorem{lemma}[theorem]{Lemma}

\theoremstyle{remark}

\theoremstyle{definition}
\newtheorem{example}[theorem]{Example}

\newcommand{\bq}{\begin{equation}}
\newcommand{\eq}{\end{equation}}
\newcommand{\beqn}{\begin{eqnarray*}}
\newcommand{\eeqn}{\end{eqnarray*}}
\newcommand{\beq}{\begin{eqnarray}}
\newcommand{\eeq}{\end{eqnarray}}

\newcommand{\rar}{\rightarrow}

\newcommand{\bc}{\begin{centre}}
\newcommand{\ec}{\end{centre}}

\newcommand{\ba}{\begin{array}}
\newcommand{\ea}{\end{array}}

\newcommand{\inp}[2]{\langle{#1},\,{#2} \rangle}

\begin{document}
\title[Circular operators]{Circular operators and their strong circularity}
\author[S. Ghara]{Soumitra Ghara}
\author[S. Kumar]{Surjit Kumar}
\author[S. Trivedi]{Shailesh Trivedi}
\address{Department of Mathematics\\
Indian Institute of Technology  Kharagpur, Midnapore-721302, India}
   \email{soumitra@maths.iitkgp.ac.in\\ ghara90@gmail.com}
\address{Department of Mathematics \\  Indian Institute of Technology Madras, Chennai  600036, India}
\email{surjit@iitm.ac.in}
\address{Department of Mathematics, Birla Institute of Technology and Science, Pilani, Pilani Campus, Vidya Vihar, Pilani, Rajasthan 333031, India}
 \email{shailesh.trivedi@pilani.bits-pilani.ac.in}

\thanks{The work of the first named  is supported by  ARG-MATRICS grant by the ANRF (File No: ANRF/ARGM/2025/001583/MTR) and INSPIRE Faculty Fellowship (DST/INSPIRE/04/2021/002555). The work of the second named author was partially supported by an IRG Grant (Ref. No. ANRF/IRG/2024/000432/MS) from the Anusandhan National Research Foundation (ANRF). The work of the third author is supported through the OPERA (FR/SCM/03-Nov-2022/MATH), funded by BITS Pilani.}

\date{}
\begin{abstract}
Circular operators have been studied extensively since the work of R. Gellar, who conjectured that every circular operator on a complex separable Hilbert space is strongly circular. In this short note, we show that circularity and strong circularity coincide for bounded operators that are finite or countably infinite direct sums of irreducible operators. This considerably narrows the search for potential counterexamples to Gellar’s conjecture. As an application, we prove that every circular operator in the Cowen–Douglas class is strongly circular. In addition, we obtain several general results on circular operators that reveal the significance of the hyper-range and the Cauchy dual.
\end{abstract}

\maketitle

\section{Introduction}

The circular symmetry of bounded linear operators on Hilbert spaces has attracted sustained attention in operator theory over the years. Circular symmetry plays a significant role in the study of spectral analysis, functional calculus, and unitary equivalence. Recall that a bounded linear operator $T$ on a complex Hilbert space $\mathcal H$ is said to be {\it circular} if for each $\lambda \in \mathbb T$, the unit circle, there exists a unitary $U_\lambda$ on $\mathcal H$ such that
\[
U_\lambda T = \lambda T U_\lambda\ \text{ for all }\ \lambda \in \mathbb T.
\]
In general, the map $\lambda \mapsto U_\lambda$ is not a homomorphism of $\mathbb T$ into $\mathcal B(\mathcal H)$, but if $\lambda \mapsto U_\lambda$ can be chosen as a strongly continuous unitary representation of $\mathbb T$, then we say that $T$ is {\it strongly circular}. 

The systematic study of circular operators was initiated by R. Gellar in \cite{G}, where he investigated strongly circular normal and subnormal operators and constructed an example of a circular operator on a non-separable Hilbert space that fails to be strongly circular. This led him to conjecture that such a pathology cannot occur on separable Hilbert spaces, that is, every circular operator on a separable Hilbert space is in fact strongly circular. A major advance towards this conjecture was later achieved by Arveson et.al. in their joint work in \cite{AHHK}, where it was shown that every irreducible circular operator on a complex separable Hilbert space is strongly circular, thereby establishing the conjecture in the irreducible case.

In the present short note, we continue this line of investigation and broadens the scope of the known positive results. We consider bounded linear operators on complex separable Hilbert spaces that admit a finite or countably infinite direct sum decomposition into irreducible operators. Our main theorem establishes that for such operators, circularity implies strong circularity. Formally, the main result of this note is stated as follows:

\begin{theorem}\label{main-thm}
Let $T$ be a bounded linear operator on a complex separable Hilbert space $\mathcal H$ such that $T = \bigoplus_{j=1}^n T_j$ with respect to the decomposition $\mathcal H = \bigoplus_{j=1}^n \mathcal H_j$, where $n$ is a positive integer with $1 \leqslant n \leqslant \infty$, and each $T_j$ is irreducible. Then $T$ is circular if and only if $T$ is strongly circular.
\end{theorem}

The proof of Theorem \ref{main-thm} utilizes advanced techniques of descriptive set theory. A key idea is to construct a subset of equivalence classes of irreducible operators and to exploit the induced group action (see Lemma \ref{grp-action}). This perspective yields a transparent mechanism that forces each irreducible summand to inherit circularity, allowing one to invoke the irreducible case of \cite{AHHK}. Although the full resolution of Gellar’s conjecture remains open, our result substantially narrows down the possible form of a counterexample, should one exist. In particular, any such example must exhibit an essentially direct integral decomposition into irreducible components (see Example \ref{direct-integral}).

In addition to the Theorem \ref{main-thm}, we develop several structural results concerning circular and left-invertible (circular) operators. These include conditions in terms of wandering subspaces, relations with the hyper-range and the Cauchy dual, and orthogonality criteria. Beyond their role in the proof of Theorem \ref{main-thm}, these results contribute to a refined understanding of how circular symmetry interacts with analyticity and the wandering subspace property of the operator in question.

We set below the notations and recall elementary facts from operator theory before proceeding towards the proof of Theorem \ref{main-thm}. The set of non-negative integers is denoted by $\mathbb N$ and the field of complex numbers is denoted by $\mathbb C$. The unit circle $\{\lambda \in \mathbb C : |\lambda| = 1\}$ is denoted by $\mathbb T$.
For a bounded linear operator $T$ on a complex separable Hilbert space $\mathcal H$, the kernel, range and the adjoint of $T$ are denoted by $\ker T$, $\operatorname{ran} T$ and $T^*$, respectively. The spectrum of $T$ is denoted by $\sigma(T)$. We say that $T$ is left-invertible if there exists a bounded linear operator $L$ on $\mathcal H$ such that $LT=I$. It turns out that $T$ is left-invertible if and only if $T^*T$ is invertible. For a left-invertible operator $T$, the Cauchy dual of $T$, denoted by $T'$, is defined as $T' := T(T^*T)^{-1}$. The hyper-range of $T$
is denoted by $T^\infty(\mathcal H)$ and is given by
\beqn
T^\infty(\mathcal H) := \bigcap_{n \in \mathbb N} T^n(\mathcal H).
\eeqn
The subspace $T^\infty(\mathcal H)$ is not closed in general, but if $T$ is left-invertible, then $T^\infty(\mathcal H)$ is a closed subspace of $\mathcal H$. For a closed subspace $\mathcal M$ of $\mathcal H$, the orthogonal projection of $\mathcal H$ onto $\mathcal M$ is denoted by $P_{\mathcal M}$. We say that $T$ has the wandering subspace property if $\mathcal H = \bigvee_{n \in \mathbb N} T^n(\ker T^*)$, where $\bigvee$ denotes the closed linear span.

For the sake of completeness, we include a few elementary facts from topological group theory. A topological group is a group equipped with a topology such that the group operations are continuous. Let $X$ be a non-empty set and $G$ be a topological group. We say that $G$ acts on $X$ if there exists a well-defined map $(g,x) \mapsto g\cdot x$ from $G \times X$ into $X$. The orbit of an element $x \in X$ is given by
\[Orb(x) = \{g\cdot x : g \in G\},\]
whereas the stabilizer of $x$ is given by
\[Stab(x) = \{g \in G : g\cdot x = x\}.\]
Note that $Stab(x)$ is a subgroup of $G$. The orbit-stabilizer theorem states that the cardinality of $Orb(x)$ is equal to the index of $Stab(x)$. In other words, there is a bijective correspondence between $Orb(x)$ and the set of all distinct left cosets of $Stab(x)$. 
 
\section{Left-invertible circular operators}

This section presents several results concerning circular and left-invertible circular operators. Beyond their relevance to the Theorem \ref{main-thm}, these results are of independent interest and contribute to a broader understanding of the subject. We begin with the following proposition which is motivated from \cite[Lemma 2.3]{GKMP}.

\begin{proposition}
Let $T$ be a bounded linear operator on a complex separable Hilbert space $\mathcal H$ with the wandering subspace property. Then $T$ is circular if and only if for each $\lambda \in \mathbb T$, there exists a unitary $V_\lambda$ on $\ker T^*$ such that the map $U_\lambda : \mathrm{span}\{T^n(\ker T^*) : n \in \mathbb N\} \rar \mathcal H$ defined as
\beq\label{U-theta-eq}
U_\lambda \Big(\sum_{j=0}^n T^j x_j \Big) = \sum_{j=0}^n (\lambda T)^j V_\lambda x_j, \quad x_j \in \ker T^*, \ n \in \mathbb N,
\eeq
extends to a unitary on $\mathcal H$.
\end{proposition}

\begin{proof}
Suppose that $T$ is circular. Let $\lambda \in \mathbb T$. Then there exists a unitary $U_\lambda$ on $\mathcal H$ such that $U_\lambda T = \lambda T U_\lambda$. Note that $U_\lambda(\ker T^*) = \ker T^*$. Set $V_\lambda = U_\lambda|_{\ker T^*}$. Then we get
\beqn
U_\lambda \Big(\sum_{j=0}^n T^j x_j \Big) = \sum_{j=0}^n (\lambda T)^j U_\lambda x_j = \sum_{j=0}^n (\lambda T)^j V_\lambda x_j, \quad x_j \in \ker T^*, \ n \in \mathbb N.
\eeqn

Conversely, for each $\lambda \in \mathbb T$, let $U_\lambda$ be as defined in \eqref{U-theta-eq} which extends to a unitary on $\mathcal H$. Let $x \in \mathcal H$. Since $T$ has the wandering subspace property, there exists a sequence $\{y_n\}_{n \geqslant 1} \subseteq \mathrm{span}\{T^n(\ker T^*) : n \in \mathbb N\}$ such that $y_n \rar x$ as $n \rar \infty$. Let 
\beqn
y_n = \sum_{j=0}^{k_n} T^j x^{(n)}_j, \quad x_j^{(n)} \in \ker T^* \ \text{ and }\ n,\, k_n \in \mathbb N. 
\eeqn
Then we obtain that
\beqn
U_\lambda T x &=& \lim_{n \rar \infty} U_\lambda T y_n = \lim_{n \rar \infty} U_\lambda \Big(\sum_{j=0}^{k_n} T^{j+1} x^{(n)}_j\Big) \overset{\eqref{U-theta-eq}}= \lim_{n \rar \infty} \lambda T \sum_{j=0}^{k_n} (\lambda T)^j V_\lambda x^{(n)}_j \\
&=& \lim_{n \rar \infty} \lambda T U_\lambda y_n = \lambda T U_\lambda x. 
\eeqn
This shows that $T$ is circular, completing the proof.
\end{proof}

Assume that $T$ has the wandering subspace property. If $\{T^n(\ker T^*)\}_{n \in \mathbb N}$ is a sequence of mutually orthogonal subspaces, then by taking $V_\lambda = I$ in \eqref{U-theta-eq}, it can be easily deduced that $T$ is circular. Conversely, if $T$ is circular, it need not follow that the subspaces $\{T^n(\ker T^*)\}_{n \in \mathbb N}$ are mutually orthogonal; see \cite{CT} for appropriate examples. Nevertheless, when $\dim\ker T^*=1$, the circularity of $T$ is equivalent to the mutual orthogonality of the family $\{T^n(\ker T^*)\}_{n \in \mathbb N}$. The following corollary, which otherwise follows from \cite[Theorem 2.5]{CY}, illustrates this fact. 

\begin{corollary}\label{cyclic-circular}
Let $T$ be a cyclic operator on a complex separable Hilbert space $\mathcal H$ with a cyclic vector $x \in \ker T^*$. Then $T$ is circular if and only if $\{T^n x\}_{n \in \mathbb N}$ is a sequence of mutually orthogonal vectors.
\end{corollary}

\begin{proof}
Note that $\ker T^*$ is one dimensional. Suppose that $T$ is circular. By the preceding proposition, for each $\lambda \in \mathbb T$, there exists $c_\lambda \in \mathbb C$ with $|c_\lambda| = 1$ such that the map $U_\lambda : \mathrm{span}\{T^n x : n \in \mathbb N\} \rar \mathcal H$ defined as
\beqn
U_\lambda \Big(\sum_{j=0}^n T^j x \Big) = \sum_{j=0}^n (\lambda T)^j c_\lambda x,  \quad n \geqslant 0,
\eeqn
extends to a unitary on $\mathcal H$. Thus, for all $m, n \in \mathbb N$, we have
\beqn
\inp{T^n x}{T^m x} = \inp{U_\lambda T^n x}{U_\lambda T^m x} = \inp{(\lambda T)^n c_\lambda x}{(\lambda T)^m c_\lambda x}= \lambda^{n-m}\inp{T^n x}{T^m x}.
\eeqn 
Consequently, $\inp{T^n x}{T^m x} = 0$ if $m \neq n$.

Conversely, if $\{T^n x\}_{n \in \mathbb N}$ is a sequence of mutually orthogonal vectors, then $\mathcal H = \oplus_{n \in \mathbb N} \mathcal H_n$, where $\mathcal H_n = \text{span}\{T^n x\}$, $n \in \mathbb N$. Taking $V_\lambda = I$ in \eqref{U-theta-eq}, it is now easy to verify that $U_\lambda$, as defined in \eqref{U-theta-eq}, extends to a unitary on $\mathcal H$. 
\end{proof}

%
%

The following proposition gives a partial converse of the fact stated prior to the foregoing corollary.

\begin{proposition}
Let $T$ be a circular operator on a complex separable Hilbert space $\mathcal H$. Suppose that $k := \dim \ker T^* < \infty$. Then there exists an orthonormal basis $\mathscr B = \{e_j : j=1 \ldots, k\}$ of $\ker T^*$ such that for all $j = 1, \ldots, k$, $\inp{T^m e_j}{T^n e_j} = 0$ for all $m, n \in \mathbb N$ with $m \neq n$.
\end{proposition}

\begin{proof}
Since $T$ is circular, for each real $\theta$, there exists a unitary $U_\theta$ on $\mathcal H$ such that $U_\theta T = e^{i\theta} T U_\theta$. Note that $U_\theta(\ker T^*) = \ker T^*$ for all real $\theta$. Let $\theta$ be such that $n\theta \neq 2\pi$ for all $n \in \mathbb Z$. Then there exists an orthonormal basis $\mathscr B = \{e_j : 1 \leqslant j \leqslant k\}$ of $\ker T^*$ such that $U_\theta e_j = \lambda_{\theta,j} e_j$ for all $j = 1, \ldots, k$ with $|\lambda_{\theta,j}| = 1$. Thus, for all $m, n \geqslant 0$ and $j = 1, \ldots, k$, we get
\beqn
\inp{T^m e_j}{T^n e_j} &=& \inp{U_\theta T^m e_j}{U_\theta T^n e_j} = \inp{e^{im\theta} T^m U_\theta e_j}{e^{in\theta}T^n U_\theta e_j}\\
&=& e^{i(m-n)\theta}\inp{T^m e_j}{T^n e_j}.
\eeqn
Hence, $\inp{T^m e_j}{T^n e_j} = 0$ if $m \neq n$ for all $m, n \in \mathbb N$ and $j = 1, \ldots, k$.
\end{proof}

For a bounded linear operator $T$ on a complex separable Hilbert space $\mathcal H$, let $\mathcal G(T)$ denote the set of all complex numbers $\lambda$ such that $\lambda T$ is unitarily equivalent to $T$. It follows from \cite[Lemma 1.1]{AHHK} that $\mathcal G(T)$ is a subgroup of the unit circle $\mathbb T$. Note that if $\mathcal G(T) = \mathbb T$, then $T$ is circular. The forthcoming proposition gives an estimate of $\mathcal G(T)$ for a left-invertible operator $T$, and relates the circularity of $T$ with that of its Cauchy dual $T'$. Moreover, it also shows that the circularity of a non-analytic left-invertible operator can be determined by its restriction on the hyper-range.

\begin{proposition}\label{Tinfinity}
Let $T$ be a left-invertible operator on a complex separable Hilbert space $\mathcal H$. Then the following statements hold:
\begin{enumerate}
\item[(i)] $T(T^\infty(\mathcal H)) = T^\infty(\mathcal H)$.
\item[(ii)] For $\lambda \in \mathcal G(T)$, let $U_\lambda$ be a unitary operator on $\mathcal H$ which intertwines $T$ and $\lambda T$. Then $T^\infty(\mathcal H)$ is a reducing subspace of $U_\lambda$ for each $\lambda \in \mathcal G(T)$. Moreover, $U_\lambda|_{T^\infty(\mathcal H)}$ is unitary on $T^\infty(\mathcal H)$.
\item[(iii)] $\mathcal G(T) \subseteq \mathcal G(T|_{T^\infty(\mathcal H)})$ and $\mathcal G(T) \subseteq \mathcal G(P_{T^\infty(\mathcal H)^\bot}T|_{T^\infty(\mathcal H)^\bot})$.
\item[(iv)] Let $T'$ be the Cauchy dual of $T$. Then $\mathcal G(T) = \mathcal G(T')$.
\item[(v)] Let $T$ be analytic and $T'$ be the Cauchy dual of $T$. Suppose that $\ker T^* = \text{span}\{x\}$. Then $T$ is circular if and only if $\{T'^n x\}_{n \in \mathbb N}$ is a sequence of mutually orthogonal vectors.
\end{enumerate}
\end{proposition}

\begin{proof}
(i)\ \ It is easy to verify that $T(T^\infty(\mathcal H)) \subseteq T^\infty(\mathcal H)$. To see the reverse inclusion, let $x \in T^\infty(\mathcal H)$. Then there exists a sequence $(x_n)_{n \geqslant 1}$ in $\mathcal H$ such that $x = T^n x_n$ for all $n \geqslant 1$. Set $y := x_1$. Then $Ty = x$. Further, $y = Lx = L T^n x_n = T^{n-1} x_n$ for all $n \geqslant 1$, where $L$ is a left-inverse of $T$. Thus $y \in T^\infty(\mathcal H)$. This proves (i).

(ii)\ \ Let $x \in T^\infty(\mathcal H)$ and $\lambda \in \mathcal G(T)$. Then there exists a sequence $(x_n)_{n \geqslant 1}$ in $\mathcal H$ such that $x = T^n x_n$ for all $n \geqslant 1$. Hence,
\beqn
U_\lambda x = U_\lambda T^n x_n = \lambda^n T^n U_\lambda x_n = T^n(\lambda^n U_\lambda x_n) \mbox{ for all } n \geqslant 1.
\eeqn
Thus $U_\lambda x \in T^\infty(\mathcal H)$, and hence, $T^\infty(\mathcal H)$ is an invariant subspace of $U_\lambda$. Similarly, using $U^*_\lambda T = \bar \lambda T U^*_\lambda$, we obtain that $T^\infty(\mathcal H)$ is an invariant subspace of $U^*_\lambda$. Thus $T^\infty(\mathcal H)$ is a reducing subspace of $U_\lambda$. The moreover part follows from the general fact that the restriction of a unitary to its reducing subspace is a unitary. This completes the verification of (ii).

(iii)\ \ Let $\lambda \in \mathcal G(T)$ and $U_\lambda$ be a unitary on $\mathcal H$ which intertwines $T$ and $\lambda T$. By virtue of (i) and (ii), let 
\beqn
T = \left(\begin{matrix}
T_1 & T_2\\
&\vspace{-.3cm}\\
0 & T_3 \end{matrix}\right) \quad \mbox{and} \quad  U_\lambda = \left(\begin{matrix}
U_{1,\lambda} & 0\\
&\vspace{-.3cm}\\
0 & U_{2, \lambda} \end{matrix}\right)
\eeqn
be the matrix representation of $T$ and $U_\lambda$ with respect to the decomposition $\mathcal H = T^\infty(\mathcal H) \oplus T^\infty(\mathcal H)^\bot$. Then $U_\lambda T = \lambda T U_\lambda$ gives that
\beqn
U_{1,\lambda} T_1 = \lambda T_1 U_{1, \lambda}\ \mbox{ and }\ U_{2,\lambda} T_3 = \lambda T_3 U_{2, \lambda}\ \mbox{ for all }\ \lambda \in \mathcal G(T).
\eeqn
In view of (ii), it follows that $\mathcal G(T) \subseteq \mathcal G(T_1)$ and $\mathcal G(T) \subseteq \mathcal G(T_3)$. Thus the conclusion in (iii) stands verified .

(iv)\ \ Let $\lambda \in \mathcal G(T)$ and $U_\lambda$ be a unitary on $\mathcal H$ such that $T = U^*_\lambda \lambda T U_\lambda$. Then we get
\beqn
T' = T(T^* T)^{-1} = (U^*_\lambda \lambda T U_\lambda) (U^*_\lambda T^* T U_\lambda)^{-1} = U^*_\lambda \lambda T' U_\lambda.
\eeqn
Thus $\lambda \in \mathcal G(T')$, and hence, $\mathcal G(T) \subseteq \mathcal G(T')$. The other way inclusion is obtained by observing that $(T')' = T$. This proves (iv).

(v)\ \ Note that as $T$ is analytic, $T'$ has the wandering subspace property \cite[Proposition 2.7]{S}. Thus $\mathcal H = \bigvee_{n \in \mathbb N} T'^n x$. In other words, $T'$ is a cyclic operator with cyclic vector $x \in \ker T^* = \ker T'^*$. Hence by Corollary \ref{cyclic-circular}, $T'$ is circular if and only if $\{T'^n x\}_{n \in \mathbb N}$ is a sequence of mutually orthogonal vectors. The desired conclusion is now immediate from (iv).
\end{proof}

Some immediate consequences of the preceding proposition are as follows.

\begin{corollary}
Let $T$ be a left-invertible operator on a complex separable Hilbert space $\mathcal H$. Then the following statements are true: 
\begin{enumerate}
\item[(i)] If $T^\infty(\mathcal H)$ is non-trivial and finite dimensional, then $\mathcal G(T)$ is finite.
\item[(ii)] If $\dim T^\infty(\mathcal H) = 1$, then $\mathcal G(T) = \mathcal G(T|_{T^\infty(\mathcal H)}) = \{1\}$.
\item[(iii)] If $T$ is analytic, then $\mathcal G(T) = \mathcal G(T|_{T^\infty(\mathcal H)})$ if and only if $T$ is circular.
\item[(iv)] If $T$ is circular, then either $T$ is analytic or $T^\infty(\mathcal H)$ is infinite dimensional.
\item[(v)] If $T$ is circular, then either $T$ has the wandering subspace property or $T'^\infty(\mathcal H)$ is infinite dimensional.
\end{enumerate}
\end{corollary}

\begin{proof}
(i)\ \ Suppose that $0 \ne k := \dim T^\infty(\mathcal H) < \infty$. Since $T$ is left-invertible, it follows from Proposition \ref{Tinfinity}(i) that $A := T|_{T^\infty(\mathcal H)}$ is an invertible $k \times k$ matrix. If $A$ is unitarily equivalent to $\lambda A$ for some $\lambda \in \mathbb T$, then $\sigma(A) = \sigma(\lambda A) = \lambda \sigma(A)$. By this, we conclude that $\mathcal G(A)$ must be finite. Consequently, by Proposition \ref{Tinfinity}(iii), $\mathcal G(T)$ is finite. The proof of (ii) is obvious in view of (i).

(iii)\ \ Suppose that $T$ is analytic. Then $T|_{T^\infty(\mathcal H)} = 0$. Since $\mathcal G(0) = \mathbb T$, we conclude that $\mathcal G(T) = \mathcal G(T|_{T^\infty(\mathcal H)})$ if and only if $T$ is circular.

(iv)\ \ Suppose that $T$ is circular but not analytic. If $T^\infty(\mathcal H)$ is finite dimensional, then by (i) $\mathcal G(T)$ is finite, which contradicts the fact that $T$ is circular. Thus, $T^\infty(\mathcal H)$ must be infinite dimensional. 

(v)\ \ Note that if $T$ is circular, then it follows from Proposition \ref{Tinfinity}(iv) that $T'$ is circular. Thus, the conclusion is now immediate in view of (iv) and \cite[Proposition 2.7]{S}. This completes the proof.
\end{proof}

\begin{corollary}
If $T$ is an isometry, then $\mathcal G(T) = \mathcal G(T|_{T^\infty(\mathcal H)})$.
\end{corollary}

\begin{proof}
Suppose that $T$ is an isometry. By the Wold decomposition, $T = T_1 \oplus S$ on $\mathcal H = T^\infty(\mathcal H) \oplus T^\infty(\mathcal H)^\bot$, where $S$ is a unilateral (possibly operator-valued) shift. In view of Proposition \ref{Tinfinity}(iii), it is enough to show that $\mathcal G(T_1) \subseteq \mathcal G(T)$. To this end, let $\lambda \in \mathcal G(T_1)$. Then there exists a unitary $W_\lambda$ on $T^\infty(\mathcal H)$ such that $W_\lambda T_1 = \lambda T_1 W_\lambda$. Since $S$ is circular \cite[Proposition 4.1]{GKT}, there exists a unitary $V_\lambda$ on $T^\infty(\mathcal H)^\bot$ such that $V_\lambda S = \lambda S V_\lambda$. Set $U_\lambda := W_\lambda \oplus V_\lambda$. Then $U_\lambda$ is a unitary on $\mathcal H$ and $U_\lambda T = \lambda T U_\lambda$. Thus $\lambda \in \mathcal G(T)$. 
\end{proof}


The inclusion $\mathcal G(T) \subseteq \mathcal G(T|_{T^\infty(\mathcal H)})$ as mentioned in the Proposition \ref{Tinfinity}(iii) is strict in general. The following example attests to this fact. 

\begin{example}
Let $\mathcal H(\kappa)$ be the reproducing kernel Hilbert space of (scalar-valued) holomorphic functions defined on the open unit disc $\mathbb D$, determined by the kernel
\[
\kappa(z,w)=\frac{1}{1-z\overline{w}}+z\overline{w}(1+z)(1+\overline{w}),
\quad z,w\in\mathbb D .
\]
Clearly $\frac{1}{1-z\overline{w}} \leqslant \kappa(z,w)$ in the sense of non-negative definiteness. Thus, by \cite[Theorem 5.1]{PR}, the Hardy space $\mathcal H^2(\mathbb D)$ is contained in $\mathcal H(\kappa)$. Also, by \cite[Theorem 3.11]{PR},
\[
z(1+z)\,\overline{w}(1+\overline{w}) \leqslant
\|z(1+z)\|_{\mathcal H^2(\mathbb D)}^{2}
\frac{1}{1-z\overline{w}} .
\]
Consequently, $\kappa(z,w) \leqslant c\,\frac{1}{1-z\overline{w}}$ for some $c>0$. Therefore by \cite[Theorem 5.1]{PR}, $\mathcal H(\kappa) \subseteq \mathcal H^2(\mathbb D)$. Hence $\mathcal H(\kappa) = \mathcal H^2(\mathbb D)$
with equivalence of norms. This yields that the operator $\mathscr M_z$ of multiplication by the coordinate function $z$ on $\mathcal H(\kappa)$
is similar to $\mathscr M_z$ on $\mathcal H^2(\mathbb D)$ with identity operator being one of the invertible intertwiners. In particular,
$\mathscr M_z$ on $\mathcal H(\kappa)$ is left-invertible, analytic and $\ker \mathscr M_z^*$ (on $\mathcal H(\kappa)$) is one dimensional subspace spanned by constant functions. Thus, $\mathscr M_z^\infty(\mathcal H(\kappa)) = \{0\}$, and hence, $\mathcal G(\mathscr M_z|_{\mathscr M_z^\infty(\mathcal H(\kappa))}) = \mathbb T$. On the other hand, $\mathscr M_z$ on $\mathcal H(\kappa)$ is not circular. In fact from above observations, $\mathscr M_z$ on $\mathcal H(\kappa)$ is a cyclic operator with cyclic vector $1 \in \ker \mathscr M_z^*$. Thus, by Corollary \ref{cyclic-circular}, $\mathscr M_z$ is circular if and only if monomials are orthogonal. Since $\kappa$ is normalized at $0$ and polynomials are dense in $\mathcal H(\kappa)$, it follows that monomials are orthogonal in $\mathcal H(\kappa)$ if and only if $\kappa$ is diagonal (see \cite{CS}). Consequently, $\mathscr M_z$ on $\mathcal H(\kappa)$ is circular if and only if $\kappa$ is diagonal. Since $\kappa$ is not diagonal, $\mathscr M_z$ on $\mathcal H(\kappa)$ is not circular.
\end{example}

\section{Proof of Theorem \ref{main-thm}}

From here onwards, we assume that $T$ is a bounded linear operator on a complex separable Hilbert space $\mathcal H$ with the following decomposition:
\begin{eqnarray*}\label{T-decom}
&& T = \bigoplus_{j=1}^n T_j\ \text{ with respect to the decomposition }\ \mathcal H = \bigoplus_{j=1}^n \mathcal H_j, 
 \text{ where $n$ is }\\ 
 && \text{ a positive integer with } 1 \leqslant n \leqslant \infty, \text{ and each } T_j\  \text{ is irreducible}. \hspace{1.5cm} (\star)
\end{eqnarray*}

A crucial step in proving Theorem \ref{main-thm} is to establish that each operator $T_j$ is circular. Once this is shown, we can apply \cite[Proposition 1.3]{AHHK} to obtain a strongly continuous unitary representation for each $T_j$. Consequently, this yields a strongly continuous unitary representation for $T$. The following lemma formalizes this argument.  

\begin{lemma}\label{first-lemma}
Let $T$ be as described by $(\star)$. If $T_j$ is circular for each $j \in \{1, \ldots, n\}$, then $T$ is strongly circular.
\end{lemma}

\begin{proof}
Suppose that $T_j$ is circular for all $j \in \{1, \ldots, n\}$. Then it follows from \cite[Proposition 1.3]{AHHK} that there exist strongly continuous unitary representations $\pi_j$ of $\mathbb{T}$ such that
\[
\pi_j(\lambda) T_j = \lambda T_j \pi_j(\lambda)\ \text{ for all }\ \lambda \in \mathbb{T}\ \text{ and }\ j = 1, \ldots, n.
\]
Consider $\pi = \oplus_{j=1}^n \pi_j$. Then $\pi$ is a strongly continuous
unitary representation of $\mathbb{T}$ such that $\pi(\lambda) T = \lambda T \pi(\lambda)$ for all $\lambda \in \mathbb{T}$. 
\end{proof}

Let $T$ be as described by $(\star)$. Define the equivalence relation $\sim$ on the set $\{T_j : 1 \leqslant j \leqslant n\}$ as
$T_i \sim T_j \text{ if and only if } T_i \text{ is unitarily equivalent to } T_j,\ 1 \leqslant i,j \leqslant n$.
Let 
\beq\label{eq-relation-1}
X := \big\{[T_{i_k}] : 1 \leqslant k \leqslant m\big\} 
\eeq
denote the set of all disjoint equivalence classes of $\{T_j : 1 \leqslant j \leqslant n\}$, where $m$ is a positive integer with $1 \leqslant m \leqslant \infty$ and $[T_{i_k}]$ denotes the equivalence class of $T_{i_k}$.

The next lemma shows that if $T$ is circular, then the circle group $\mathbb T$ acts on $X$ and the stabilizer subgroup of each element of $X$ is an analytic subset of $\mathbb T$. For that, we recall a few basic facts from descriptive set theory. The reader is referred to \cite{K} for more details.

A subset of a topological space is said to be a {\it meager set} if it is a countable union of nowhere dense subsets. A  subset $A$ of a topological space is said to have the {\it Baire property} if there exists an open subset $U$ such that the symmetric difference $A \Delta U$ is a meager set. Recall that a separable complete metric space is called as {\it Polish space}. A subset $A$ of a Polish space $E$ is called {\it analytic} if there is a Polish space $F$ and a continuous function $f : F \to E$ such that $f(F) = A$. We quote here two celebrated results which will be used in the proof of Theorem \ref{main-thm}. The first one is due to Lusin and Sierpi\'nski which can be found in \cite[Theorem 21.6, Pg 153]{K} whereas the second one is attributed to Pettis as given in \cite[Theorem 9.9, Pg 61]{K}. 

\begin{theorem}[Lusin-Sierpi\'nski]\label{LST}
All analytic subsets of a Polish space have the Baire property.
\end{theorem}     

\begin{theorem}[Pettis]\label{PT}
Let $A$ be a non-meager susbet of a topological group $G$. If $A$ has the Baire property, then the set $A^{-1}A = \{x^{-1}y : x, y \in A\}$ contains an open neighbourhood of $1$.
\end{theorem}

The following lemma is an intermediate step of the proof of Theorem \ref{main-thm}.

\begin{lemma}\label{grp-action}
Let $T$ be as described by $(\star)$ and let $X$ be as described by \eqref{eq-relation-1}. If $T$ is circular, then the following are true:
\begin{enumerate}
\item[(i)] The group $\mathbb T$ acts on $X$.
\item[(ii)] For each $k \in  \{1, \ldots, m\}$, the stabilizer subgroup $Stab([T_{i_k}])$ of $[T_{i_k}]$ is an analytic subset of $\mathbb T$. 
\end{enumerate}  
\end{lemma}

\begin{proof}
Suppose that $T$ is circular. Let $\lambda \in \mathbb T$. Since $T$ is unitarily equivalent to $\lambda T$, it follows from \cite[Theorem 3.1]{FJW} (see also \cite[Proposition 3.4]{FJW}) that $\lambda T_{i_k}$ is unitarily equivalent to $T_{\sigma(i_k)}$ for some permutation $\sigma$ on $\{i_k : k = 1, \ldots, m\}$. Thus, the circle group $\mathbb T$ acts on $X$ via the following action:
\[(\lambda, [T_{i_k}]) \mapsto [T_{\sigma(i_k)}]\ \text{ for all }\ \lambda \in \mathbb T \ \text{ and }\ k = 1, \ldots, m.\]
This proves (i).

To see (ii), let $[T_{i_k}]$ be an arbitrary but fixed element of $X$. Let $\mathcal U(\mathcal H_{i_k})$ denote the set of all unitaries on $\mathcal H_{i_k}$. Since $\mathbb T$ with usual topology and $\mathcal U(\mathcal H_{i_k})$ with strong operator topology are Polish spaces, the product space $\mathbb T \times \mathcal U(\mathcal H_{i_k})$ is also a Polish space in the product topology. Now consider the subset $Y$ of $\mathbb T \times \mathcal U(\mathcal H_{i_k})$ defined as follows:
\[Y = \{(\lambda, U) : U T_{i_k} = \lambda T_{i_k} U\}.\]
Then $Y$ is closed. In fact, if $\{(\lambda_n, U_n)\}_{n \geqslant 1}$ is any sequence in $Y$ which converges to $(\lambda, U) \in \mathbb T \times \mathcal U(\mathcal H_{i_k})$, then $\lambda_n \to \lambda$ and $U_n \to U$ (SOT) as $n \to \infty$. Thus, for all $h \in \mathcal H_{i_k}$, we get
\beqn
U T_{i_k} h = \lim_{n \to \infty} U_n T_{i_k} h = \lim_{n \to \infty} \lambda_n T_{i_k} U_n h = \lambda T_{i_k} U h.
\eeqn
This verifies the fact that $Y$ is closed. Consequently, $Y$ is a Polish space. Since $\mathbb T \times \mathcal U(\mathcal H_{i_k})$ is equipped with the product topology, the projection map $\pi_1 : \mathbb T \times \mathcal U(\mathcal H_{i_k}) \to \mathbb T$ is continuous. Observe that $\pi_1(Y) = Stab([T_{i_k}])$. Thus, $Stab([T_{i_k}])$ is an analytic subset of $\mathbb T$. This completes the proof.
\end{proof}

\begin{proof}[Proof of Theorem \ref{main-thm}]
Suppose that $T$ is circular. Then by Lemma \ref{grp-action}(i), the circle group $\mathbb T$ acts on $X$, where $X$ is as described by \eqref{eq-relation-1}. Let $[T_{i_k}]$ be an arbitrary but fixed element of $X$. By the orbit-stabilizer theorem, there is a bijective correspondence between $Orb([T_{i_k}])$ and the set of all distinct left cosets of $Stab([T_{i_k}])$. Consequently, the set of all distinct left cosets of $Stab([T_{i_k}])$ is countable. Let these distinct left cosets be denoted by $\lambda_\ell Stab([T_{i_k}])$, $\ell = 1, \ldots, p$, where $\lambda_\ell \in \mathbb T$ and $p$ is a positive integer with $1 \leqslant p \leqslant \infty$. Thus, we have
\[\mathbb T = \bigsqcup_{l=1}^p \lambda_\ell Stab([T_{i_k}]).\]
If each of the coset $\lambda_\ell Stab([T_{i_k}])$ is meager, then $\mathbb T$ would be meager which contradicts the Baire category theorem. Therefore, there exists $\ell_0 \in \{1, \ldots, p\}$ such that $\lambda_{\ell_0} Stab([T_{i_k}])$ is non-meager. Since multiplication by a fixed element of $\mathbb T$ is a homeomorphism of $\mathbb T$, it follows that $Stab([T_{i_k}])$ is a non-meager subset of $\mathbb T$. Thus, by Lemma \ref{grp-action}(ii), we get that $Stab([T_{i_k}])$ is a non-meager analytic subset of $\mathbb T$. Since $Stab([T_{i_k}])$ is a subgroup of $\mathbb T$, by Theorems \ref{LST} and \ref{PT} we obtain that $Stab([T_{i_k}])$ contains an open neighbourhood of $1$. Consequently, $Stab([T_{i_k}])$ is an open subgroup of $\mathbb T$. Since in a topological group every open subgroup is closed, we get that $Stab([T_{i_k}])$ is both open and closed. The connectedness of $\mathbb T$ yields that $Stab([T_{i_k}]) = \mathbb T$. Consequently, $T_{i_k}$ is circular. Since $[T_{i_k}]$ was arbitrarily chosen element of $X$, in view of Lemma \ref{first-lemma} we conclude that $T$ is strongly circular. The converse is obvious, and thus, completes the proof.
\end{proof}

As an application of Theorem \ref{main-thm}, we show that every operator in the Cowen-Douglas class $B_n(\Omega)$ is circular if and only if it is strongly circular. Recall that for a connected open subset $\Omega$ of $\mathbb C$ and a finite positive integer $n$, a bounded linear operator $T$ on a complex separable Hilbert space $\mathcal H$ is said to belong to the Cowen-Douglas class $B_n(\Omega)$ if the following conditions are satisfied:
\begin{enumerate}
\item[(i)] $\Omega \subseteq \sigma(T)$;
\item[(ii)] $\text{ran}(T-w) = \mathcal H$ for all $w \in \Omega$;
\item[(iii)] $\bigvee_{w \in \Omega} \ker (T-w) = \mathcal H$;
\item[(iv)] $\dim \ker (T-w) = n$ for all $w \in \Omega$.
\end{enumerate}  

\begin{corollary}
If $T \in B_n(\Omega)$, then $T$ is circular if and only if $T$ is strongly circular.
\end{corollary}

\begin{proof}
Suppose that $T \in B_n(\Omega)$. Let $V^*(T)$ denote the von Neumann algebra consisting of operators commuting with both $T$ and $T^*$. Then by \cite[Corollary 3.7]{CDG}, $V^*(T)$ is finite dimensional. Hence by \cite[Corollary 3.2]{FJW}, $T$ can be written as a direct sum of finitely many irreducible operators. Consequently, by Theorem \ref{main-thm}, $T$ is circular if and only if $T$ is strongly circular.
\end{proof}

Unlike the case of Theorem \ref{main-thm}, circularity does not, in general, pass to the irreducible components of a direct integral decomposition. The following example attests to this phenomenon (the reader is referred to \cite{D} for a detailed study on the theory of direct integral). 
\begin{example}\label{direct-integral}
Consider the Hilbert space $L^2(\mathbb T)$ of square-integrable functions with respect to the normalized Lebesgue measure
\[
d\mu(z)=\frac{d\theta}{2\pi}, \quad z=e^{i\theta}\in\mathbb T.
\]
Consider the operator $\mathscr M_z$ of multiplication by the coordinate function $z$ on $L^2(\mathbb T)$. Note that $\mathscr M_z$ is a circular operator; in fact, it is strongly circular. Indeed, for every $\lambda\in\mathbb T$, the unitary operator
\[
(U_\lambda f)(z)=f(\lambda z), \qquad f\in L^2(\mathbb T),
\]
satisfies $U_\lambda \mathscr M_z = \lambda \mathscr M_{z} U_\lambda$ and the map $\lambda \mapsto U_\lambda$ is a strongly continuous unitary representation of $\mathbb T$. It follows from Wiener's Theorem \cite[Theorem 1.2.1]{N} that $\mathscr M_z$ cannot be decomposed into countably many irreducible operators. In fact, $\mathscr M_z$ admits the direct integral decomposition
\[
\mathscr M_z
=\int_{\mathbb T}^{\oplus} T_z\, d\mu(z)
\quad \text{on} \quad
L^2(\mathbb T)
=
\int_{\mathbb T}^{\oplus} H_z\, d\mu(z),
\]
where $H_z=\mathbb C$ and $T_z:H_z\to H_z$ is defined as $T_z(1)=z
$. Clearly, none of the operators $T_z$ is circular.
\end{example}

\medskip \textit{Acknowledgment}.
The authors convey their sincere thanks to Sameer Chavan for his several helpful suggestions.

\end{document}